\documentclass[12pt,twoside]{amsart}
\usepackage[latin1]{inputenc}
\usepackage{times}
\usepackage{amsthm}
\usepackage{amssymb, verbatim}
\usepackage{amsmath}
\usepackage{mathrsfs}

\topmargin = - 40 pt %distanza dal margine superiore 0= un bel po` negativo diminuisce
\usepackage[all]{xy}
\makeindex

\newtheorem{teo}{Theorem}[section]
\newtheorem{cor}[teo]{Corollary}
\newtheorem{lem}[teo]{Lemma}
\newtheorem{prop}[teo]{Proposition}

\theoremstyle{remark}
\newtheorem{oss}[teo]{Remark}
\theoremstyle{definition}
\newtheorem{defi}[teo]{Definition}
\theoremstyle{definition}
\newtheorem{ex}[teo]{Example}

\theoremstyle{conjecture}
\newtheorem{conj}{Conjecture}
\newtheorem*{thm}{Theorem}

\numberwithin{equation}{section}

\DeclareMathOperator{\Can}{Can}
\DeclareMathOperator{\T}{T}
\DeclareMathOperator{\End}{End}
\DeclareMathOperator{\h}{h}
\renewcommand{\phi}{\varphi}
\renewcommand{\epsilon}{\varepsilon}
\newcommand{\de}{\partial }
\DeclareMathOperator{\W}{F}
\renewcommand{\S}{\mathcal{S}}
\DeclareMathOperator{\FF}{\mathbf{F}}
\DeclareMathOperator{\HK}{\mathbf{HK}}
\newcommand{\cb}{\preceq}%%% comes before
\newcommand{\ca}{\succeq}%%% comes after
\newcommand{\qsubw}{\leq}%%% is subword of
\date{\today}

\begin{document}
\title[A graph-dynamical interpretation of Kiselman's semigroups]{A graph-dynamical interpretation of Kiselman's semigroups}
%{Sequential Dynamical Systems on finite acyclic \\oriented graphs and Hecke-Kiselman monoids%$\Gamma_n$}
\author[E.~Collina]{Elena Collina}
\email{collina@dmmm.uniroma1.it, elena.collina@bancaditalia.it}
\address{Dipartimento di Scienze di Base e Applicate per l'Ingegneria, Universit\`a di
Roma ``La Sapienza''\\
Via Antonio Scarpa 14/16 -- 00161 Rome, Italy}
\address{Servizio Gestione Rischi Finanziari, Banca d'Italia\\
Via Nazionale 91 -- 00184 Rome, Italy}
\author[A.~D'Andrea]{Alessandro D'Andrea}
\email{dandrea@mat.uniroma1.it}
\address{Dipartimento di Matematica, Universit\`a di
Roma ``La Sapienza''\\
P.le Aldo Moro, 5 -- 00185 Rome, Italy}
%\thanks{}

\begin{abstract}
A Sequential Dynamical System (SDS) is a quadruple $(\Gamma, S_i,f_i,w)$ consisting of a (directed) graph $\Gamma=(V,E)$,
each of whose vertices $i\in V$ is endowed with a finite set state $S_i$ and an update function $f_i: \prod_{j, i \to j} S_j \to S_i$ --- we call this structure an {\em update system} --- and a word $w$ in the free monoid over $V$, specifying the order in which update functions are to be performed. Each word induces an evolution of the system and in this paper we are interested in the dynamics monoid, whose elements are all possible evolutions.

When $\Gamma$ is a directed acyclic graph, the dynamics monoid of every update system supported on $\Gamma$ naturally arises as a quotient of the Hecke-Kiselman monoid associated with $\Gamma$. In the special case where $\Gamma = \Gamma_n$ is the complete oriented acyclic graph on $n$ vertices, we exhibit an update system whose dynamics monoid coincides with Kiselman's semigroup $\mathrm{K}_n$, thus showing that the defining Hecke-Kiselman relations are optimal in this situation. We then speculate on how these results may extend to the general acyclic case.
\end{abstract}
\maketitle
\tableofcontents %(per fare l'indice)

\section{Introduction}
In this paper, we show how the recently defined notion of Hecke-Kiselman monoid finds a natural realization in the combinatorial-computational setting of Sequential Dynamical Systems.

Let $Q$ be a mixed graph, i.e., a simple graph with at most one connection between each pair of distinct vertices; connections can be either oriented (arrows) or non-oriented (edges). In \cite{MR2821178}, Ganyushkin and Mazorchuk associated with $Q$ a semigroup $\HK_Q$ generated by idempotents $a_i$ indexed by vertices of $Q$, subject to the following relations
\begin{itemize}
\item $a_ia_j=a_ja_i$, if $i$ and $j$ are not connected;
\item $a_ia_ja_i=a_ja_ia_j$, if $(\xymatrix@-1pc{i\ar@{-}[r]&j}) \in Q$, i.e., $i$ and $j$ are connected by an edge;
\item $a_ia_j=a_ia_ja_i=a_ja_ia_j$, if $(\xymatrix@-1pc{i\ar[r]&j}) \in Q$, i.e., $i$ and $j$ are connected by an arrow from $i$ to $j$.
\end{itemize}
The semigroup $\HK_Q$ is known as the \textit{Hecke-Kiselman monoid} attached to $Q$.

For the two extremal types of mixed graphs --- graphs, where all sides are edges, and oriented graphs, in which all sides are arrows --- Hecke-Kiselman monoids are well understood: when $Q$ is an oriented graph with $n$ vertices and no oriented cycles, then $\HK_Q$ is isomorphic to a quotient of Kiselman's semigroup $\mathrm{K}_n$ \cite{MR1881029, MR2561084}, which is known to be finite \cite{MR2561084}.

On the other hand, when $Q$ has only unoriented edges, $\HK_Q$ is finite if and only if $Q$ is a (finite) disjoint union of finite simply laced Dynkin diagrams, and the corresponding semigroup is then variously known as \textit{Springer-Richardson}, \textit{$0$-Hecke}, or \textit{Coxeter monoid} attached to $Q$. The problem of characterizing mixed graphs inducing a finite Hecke-Kiselman monoid in the general case seems to be difficult and only very partial results are known \cite{small}. The study of (certain quotients of) Hecke-Kiselman monoids and their representations has also attracted recent interest, see for instance \cite{forsberg, grensing, grensingmazorchuk}.

The choice of a simple (i.e., without loops or multiple edges) %oriented
graph $\Gamma$ is also one of the essential ingredients in the definition of a Sequential Dynamical System (SDS). The notion of SDS has been introduced by Barrett, Mortveit and Reidys \cite{BMR1, BMR2, BR} in order to construct a mathematically sound framework for investigating computer simulations; this structure has found wide applicability in many concrete situations, cf. \cite{MR2357144} and references therein. SDS on a directed acyclic graph $\Gamma$ are related to Hecke-Kiselman monoids in that (see Proposition \ref{P: I, II, III} below) the so-called $\Gamma$-local functions \cite{MR2357144} satisfy the relations listed in the presentation of $\HK_\Gamma$; in other words, the evaluation morphism mapping each word (or update schedule) in the alphabet $V$ to the corresponding composition of $\Gamma$-local functions factors through $\HK_\Gamma$.

One is naturally led to wonder whether $\HK_\Gamma$ is the smallest quotient through which all of the above evaluation morphisms must factor, or additional universal relations among $\Gamma$-local functions may be found. Henceforth, $\Gamma_n = (V, E)$ will denote the oriented graph where $V = \{1, \dots, n\}$ and $(i, j) \in E$ if and only if $i < j$. In this paper, we show that $\HK_\Gamma$ is optimal in the special case $\Gamma = \Gamma_n$, by proving Theorem \ref{T: Can(w)=[p_n, dots [p_2, p_1]]}, which immediately implies the following statement.
\begin{thm}
There exists an update scheme $\S_n^\star = (\Gamma_n, S_i, f_i)$ such that the associated evaluation morphism factors through no nontrivial quotient of $\HK_{\Gamma_n} $.
\end{thm}
We believe that the same claim holds for every finite directed acyclic graph and mention in Section \ref{conclude} some evidence in support of this conjecture.

The paper is structured as follows. In Section \ref{SDS} we recall the definition of SDS, define the dynamics monoid of an update system supported on an oriented graph $\Gamma$, and show that it is a quotient of the Hecke-Kiselman monoid $\HK_\Gamma$ as soon as $\Gamma$ has no oriented cycles. In Sections \ref{combo} and \ref{canon} we list the results on Kiselman's semigroup $\mathrm{K}_n$ that are contained in \cite{MR2561084} and draw some useful consequences. Section \ref{join} introduces the {\em join operation}, which is the key ingredient in the definition of the update system $\S_n^\star$, which is given in Section \ref{sstella}. The rest of the paper is devoted to the proof of Theorem \ref{main}.

\section{Sequential dynamical systems}\label{SDS}

An \emph{update system} is a triple $\S=\left(\Gamma, (S_i)_{i \in V}, (f_i)_{i \in V}\right)$ consisting of
\begin{enumerate}
\item a \emph{base graph} $\Gamma=(V,E)$,
which is a finite directed graph, with $V$ as vertex set and $E\subseteq V\times V$ as edge set; we will write
$i\to j$ for $(i,j)\in E$. The \emph{vertex neighbourhood} of a given vertex $i \in V$ is the subset
\[
x[i]=\{j \colon i\to j\}.
\]
\item a collection $S_i, i \in V,$ of finite sets of {\em states}. We denote by $S=\prod_{i\in V}S_i$ the family of
all the possible \emph{system states}, i.e., $n$-tuples $\mathbf{s} = (s_i)_{i \in V}$, where $s_i$
belongs to $S_i$ for each vertex $i$. The \emph{state neighbourhood} of $i\in V$ is $S[i] = \prod_{j \in x[i]} S_j$, and the restriction of $\mathbf{s} = (s_j)_{j \in V}$ to $x[i]$ is denoted by
\[
\mathbf{s}[i]=(s_j)_{j\in x[i]}\in S[i].
\]
 \item for every vertex $i$, a \emph {vertex (update) function}
\begin{eqnarray*}
f_i\colon S[i] &\to& S_i,
\end{eqnarray*} computes the new state value on vertex $i$ as a function of its state neighbourhood.
In particular, if $x[i]$ is empty, then $f_i$ is a constant $t\in S_i$ and we will write $f_i\equiv t$.
Each vertex function $f_i$ can be incorporated into a \emph{$\Gamma$-local function} $\FF_i\colon S \to S$ defined as
\begin{eqnarray*}
\FF_i(\mathbf{s})=\mathbf{t}=(t_j)_{j \in V}, \mbox{ where }t_j=
\left\{
\begin{array}{cl}
s_j, & \mbox{if } i\neq j \\
f_i(\mathbf{s}[i]), & \mbox{if } i=j
\end{array}
\right.
\end{eqnarray*}
\end{enumerate}

An SDS is an update system $\S$ endowed with
\begin{enumerate}
\item[(4)] an {\em update schedule}, i.e., a word $w=i_1 i_2 \dots i_k$ in the free monoid $\W(V)$ over the alphabet $V$ (from now on, we will often abuse the notation denoting both the alphabet and the vertex set with the same letter $V$, and both letters and vertices with the same symbols).  The update schedule $w$ induces a dynamical system map (SDS map), or an {\em evolution} of $S$,
$\FF_w\colon S \to S$, defined as
\[
\FF_w=\FF_{i_1}\FF_{i_2}\dots \FF_{i_k}.
\]
\end{enumerate}
\begin{oss}
As the graph $\Gamma$ sets up a dependence relation between nodes under the action of update functions, it makes sense to allow $\Gamma$ to possess self-loops, and arrows connecting the same vertices but going in opposite directions. However, we exclude the possibility of multiple edges between any two given vertices. Notice, however, that all SDS of interest in this paper will be supported on directed {\em acyclic} graphs, thus excluding in particular the possibility of self-loops.
\end{oss}
Denote by $\End(S)$ the set of all maps $S \to S$, with the monoid structure given by composition. Then, the $\Gamma$-local functions $\FF_i, i \in V,$ generate a submonoid of $S$ which we denote by $D(\S)$. 
The monoid $D(\S)$ is the image of the natural homomorphism
\begin{eqnarray*}
\FF \colon \W(V) &\to& \End(S) \\
w &\mapsto & \FF_w.
\end{eqnarray*}
mapping each update schedule $w$ to the corresponding evolution $\FF_w$; in particular, we denote by $\FF_{\star}$ the identity map, induced by the empty word $\star$.

Once an underlying update system has been chosen, our goal is to understand the monoid structure of $D(\S)$.

\begin{ex}\label{Ex: A_1}
Let $\Gamma=(\{i\},\emptyset)$ be a Dynkin graph of type $A_1$. It has only one vertex $i$, so there is only
one vertex function $f_i$, which is constant as there are no arrows starting in $i$. The
system dynamics monoid
$D(\S)=\{\FF_{\star},\FF_i\}$ contains exactly two
elements, as soon as $|S_i|>1$. If $|S_i| = 1$, then $D(\S) = \{\FF_{\star}\}$.
\end{ex}

\begin{ex}\label{Ex: A_2}
 Let $\Gamma$ be the graph
\[
\xymatrix@1{
i\ar[r] & \,j
}
\]
Let us consider $S_i=\{0, 1, 2\}$, $S_j=\{0, 1\}$. Set up an update system on $\Gamma$ by requiring that $f_i\colon S[i]=S_j\to S_i$ acts as $f_i(s)=s+1$ and that $f_j\equiv 1$ . Evolutions induced by words
$\star$ (the empty word), $i, j, ij$ and $ji$  on $S = S_i \times S_j$ all differ from each other, as they take different values on $(0,0)$:
\begin{eqnarray*}
\FF_\star (0,0)& = & (0,0)\\
\FF_i(0,0)& = & (1,0)\\
\FF_j(0,0)& = & (0,1)\\
\FF_{ij}(0,0)=\FF_i \FF_j (0,0)=\FF_i (0,1)& = & (2,1)\\
\FF_{ji}(0,0)=\FF_j \FF_i (0,0)=\FF_j (1,0)& = & (1,1).
\end{eqnarray*}
Both $\FF_i$ and $\FF_j$ are idempotent. Moreover, it is easy to see that
\[
	\FF_{iji}=\FF_{jij}=\FF_{ij},
\]	
as $\FF_i \FF_j \FF_i = \FF_j \FF_i \FF_j = \FF_i \FF_j$,
so that $\FF_i \mapsto a_1, \FF_j \mapsto a_2$ extends to an isomorphism from $D(\S)$ to Kiselman's semigroup $\mathrm{K}_2$%, defined in \cite{MR1881029} as
%\begin{eqnarray*}
%\mathrm{K}_2=\langle a_1,a_2 \colon && a_1^2=a_1,\,a_2^2=a_2\\
%									&& a_1 a_2 a_1=a_2 a_1 a_2 = a_1 a_2 \rangle.
%\end{eqnarray*}
\begin{equation*}
\mathrm{K}_2=\langle a_1, a_2 \,|\, a_1^2=a_1,\,a_2^2=a_2,\,a_1 a_2 a_1=a_2 a_1 a_2 = a_1 a_2\rangle.
\end{equation*}
\end{ex}
These examples are instances of the following statement.
\begin{prop}\label{P: I, II, III}
Let $\S=(\Gamma, S_i, f_i)$ be an update system defined on a directed acyclic graph $\Gamma$. Then the $\Gamma$-local functions $\FF_i$ satisfy:
\begin{enumerate}
\item[$(i)$] $\FF_i^2=\FF_i$,  for every $i\in V$;
									
\item[$(ii)$] $\FF_i\FF_j \FF_i=\FF_j \FF_i \FF_j = \FF_i \FF_j$, if $i\to j$ (and hence, $j \nrightarrow i$);

\item[$(iii)$] $\FF_i \FF_j=\FF_j \FF_i$, if $i$ and $j$ are not connected.
\end{enumerate}
\end{prop}
\begin{proof}
Every $\Gamma$-local function $\FF_i$ only affects the vertex state $s_i$. As $\FF_i(\mathbf s)$ only depends on $\mathbf s[i]$, and $i \notin x[i]$, then each $\FF_i$ is idempotent.

For the same reason, it is enough to check $(ii)$ and $(iii)$ on a graph with two vertices $i \neq j$. If they are not connected, then both $f_i$ and $f_j$ are constant, hence $\FF_i$ and $\FF_j$ trivially commute. If there is an arrow $i \to j$, then $\FF_i(a_i, a_j) = (f_i(a_j), a_i)$, whereas $\FF_j(a_i, a_j) = (a_i, t)$, since $f_j \equiv t$ is a constant. Then it is easy to check that the compositions $\FF_i \FF_j , \FF_i \FF_j \FF_i, \FF_j \FF_i \FF_j$ coincide, as they map every element to $(f_i(t), t)$.
\end{proof}

These relations are remindful of those in the presentation of a Hecke-Kiselman monoid.
\begin{defi}
Let $\Gamma=(V,E)$ be a finite directed acyclic graph.
The \emph{Hecke-Kiselman monoid} associated with $\Gamma$ is defined as follows
\begin{eqnarray}
\mathbf{HK}_{\Gamma}=\langle a_i, i\in V \,|\, && a_i^2=a_i, \mbox{ for every }i\in V;\nonumber
\\
									&& a_i a_j a_i=a_j a_i a_j = a_i a_j, \mbox{ for }i\rightarrow j;\label{2}\\
                                    && a_i a_j=a_j a_i, \mbox{ for }i \nrightarrow j, \mbox{ and } j\nrightarrow i\label{3}\rangle
\end{eqnarray}
\end{defi}
%The relation in \eqref{2} is a skew version of the ordinary braid relation, whereas \eqref{3} imposes commutativity of disconnected vertices.
This structure has been first introduced in \cite{MR2821178} for a finite mixed graph, i.e., a simple graph (without loops or multiple edges) in which edges can be either oriented or unoriented: there, an unoriented edge $(i,j)$ is used to impose the customary braid relation $ a_i a_j a_i=a_j a_i a_j$.

If $\S=(\Gamma, S_i, f_i)$ is an update system on a finite directed acyclic graph $\Gamma=(V, E)$, then Proposition \ref{P: I, II, III} amounts to claiming that the evaluation homomorphism $\FF\colon \W(V) \to \End(S)$ factors through the Hecke-Kiselman monoid $\mathbf{HK}_{\Gamma}$.

Our case of interest is when the graph $\Gamma = \Gamma_n$ is the complete graph on $n$ vertices, where the orientation
is set so that $i\rightarrow j$ if $i< j$. In this case, the semigroup
$\mathbf{HK}_{\Gamma_n}$ coincides with Kiselman's semigroup
$\mathrm{K}_n$, as defined in \cite{MR2561084}.  The monoid $\mathrm{K}_n$, however, only reflects immediate pairwise interactions between vertex functions. One may, in principle, wonder if $\mathrm{K}_n$ is indeed the smallest quotient of $\W(V)$ through which evaluation maps $\FF$ factor, or additional identities may be imposed that reflect higher order interactions.

In this paper we will exhibit an update system $\S_n^\star$, defined on the graph $\Gamma_n$, whose dynamics monoid is isomorphic to $\mathrm{K}_n$. In other words, we will show that $\mathrm{K}_n\to D(\S_n^\star)$ is indeed an isomorphism, once suitable vertex functions have been chosen.

\section{Combinatorial definitions}\label{combo}
Let $\W(A)$ be the free monoid over the alphabet $A$ and denote by $\star$ the empty word. Recall that, for every subset $B\subseteq A$, the submonoid $\langle B\rangle\subseteq \W(A)$ is identified with the free monoid $\W(B)$.

\begin{defi}%(\cite{MR2561084})
Let $w \in \W(A)$. We define
\begin{itemize}
\item \emph{subword of} $w$ to be a substring of consecutive letters of $w$;
\item \emph{quasi-subword of} $w$ to be an ordered substring $u$ of not necessarily consecutive letters of $w$.
\end{itemize}
We will denote the relation of being a quasi-subword by $\qsubw$, so that
\[
v \qsubw w
\]
if and only if $v$ is a quasi-subword of $w$.
\end{defi}
Obviously, every subword is a quasi-subword. Also notice that $v \qsubw w$ and $w \qsubw v$ if and only if $v = w$.
\begin{ex}
Set $w=acaab\in \W(\{a,b,c\})$; then
\begin{itemize}
\item $aab$ is a subword (hence a quasi-subword) of $w$;
\item $aaa$ is a quasi-subword of $w$ which is not a subword;
\item $abc$ is neither a subword nor a quasi-subword of $w$.
\end{itemize}
Trivial examples of subwords of $w$ are the empty word $\star$ and $w$ itself.
\end{ex}
\begin{defi}
Let $w\in \W(A)$. Then
\begin{itemize}
\item if $w$ is non-empty, the \emph{head} of $w$, denoted $\h(w) \in A$, is the leftmost letter in $w$;
\item
if $a\in A$, then the $a$-{\em truncation} $\T_a w\in \W(A)$ of $w$ is the longest (non-empty) suffix of $w$ with head $a$, or the empty word in case $a$ does not occur in $w$.
\end{itemize}
Similarly, if $I \subset A$, we denote by $\T_I w$ the longest (non-empty) suffix of $w$ whose head lies in $I$, or the empty word in case no letter from $I$ occurs in $w$.
\end{defi}
The following observations all have trivial proofs.
\begin{oss}\label{O: w=w'T_i(w)}
\mbox{}
\begin{enumerate}
%\item[$(i)$]If $w\in \W(A)$ does not contain any occurrence of $a\in A$, then $\T_a w=\star$.
\item[$(i)$]If $w\in \W(A)$ does not contain any occurrence of $a\in A$, then
\[
\T_a (ww')=\T_a (w'),
\]
for every $w'\in  \W(A)$.
\item[$(ii)$] For every $w\in \W(A)$ and $a\in A$, one may (uniquely) express $w$ as
\[
w=w'\T_a w
\]
where $w'\in \langle A\setminus \{a\}\rangle$.
\item[$(iii)$] If $w\in \W(A)$ contains some occurrence of $a\in A$, then
\[
\T_a (ww')=(\T_a w) w',
\]
for every $w'\in  \W(A)$.
\end{enumerate}
\end{oss}

\begin{oss}\label{O: T_a < T_b o vice versa}
\mbox{}
\begin{enumerate}
\item[$(i)$]If $w\in \W(A),$ and $a,b\in A$, then either $\T_a w$ is a suffix of $\T_b w$, or vice versa.
\item[$(ii)$] If $\T_b w$ is a suffix of $\T_a w$ for all $b\in A$ , then $\h(w) = a$, hence $\T_a w = w$.
\item[$(iii)$] $\T_I w = \T_a w$ for some $a \in I$, and $\T_b w$ is a suffix of $\T_I w$ for every $b \in I$.
\end{enumerate}
\end{oss}

\begin{defi}
Given $I\subseteq A$, the \emph{deletion morphism} is the unique semigroup homomorphism satisfying
\begin{eqnarray*}
\de_I \colon  \W(A) &\to& \W(A\setminus I)\subseteq  \W(A)\\
a_i &\mapsto & \star \mbox{,  for }i\in I\\
a_j &\mapsto & a_j \mbox{, for }j\notin I.
\end{eqnarray*}
It associates with any $w\in \W(A)$ the longest quasi-subword of $w$ containing no occurrence of letters from $I$.
\end{defi}
\begin{oss}\label{O: de}
For every $I,J\subseteq A$,
\[
\de_I\,\de_J=\de_{I\cup J}
\]
\end{oss}

%
%\begin{lem}\label{L: de-trunc}
%If $a \notin I\subset A$, then
%\[
%\de_I\T_a w=\T_a\de_I w
%\]
%for every $w\in \W(A)$.
%\end{lem}
%\begin{proof}
%Write $w=w' \T_a w$. Then, by Remark \ref{O: w=w'T_i(w)}%%$(iii)$,
%$(i)$,
%\begin{eqnarray}
%\T_a \de_I w& = & \T_a\de_I\left(w'\, \T_a w\right)\nonumber\\
%& = &\T_a\left(\de_I w' \,\de_I\T_a w \right)\nonumber\\
%& = &\T_a  \de_I\T_a w\label{2.1}\\
%& = & \de_I\T_a w \label{2.2}
%\end{eqnarray}
%In \eqref{2.1} we use that $a$ does not occur in $\de_I w'$, as, by definition of $w'$, $a$ does %not occur in $w'$;
%then \eqref{2.2} holds because either the head of
%$\T_a w$ is $a$ or $\T_a w = \star$, hence the same holds after applying $\de_I$, as $a\notin I$.
%\end{proof}

\section{Kiselman's semigroup and canonical words over the complete graph}\label{canon}
In this section, we recall results from \cite{MR2821178} and draw some further consequences.
%Henceforth, when $h\leq k$ are natural numbers, we will denote by $[h,k] = \{h, h+1, \dots, k\}$ the corresponding segment.
Choose an alphabet $A=\{a_i, 1 \leq i\leq n\}$; Kiselman's semigroup $\mathrm{K}_n$ has the presentation
\begin{equation*}
\mathrm{K}_n=\langle a_i\in A \,|\, a_i^2=a_i \mbox{ for all }i; \,\, a_i a_j a_i=a_j a_i a_j = a_i a_j \mbox{ for }i<j \rangle.
\end{equation*}
%\begin{eqnarray*}
%\mathrm{K}_n=\langle a_i\in A \colon && a_i^2=a_i \mbox{ for every }i,\\
%														&& a_i a_j a_i=a_j a_i a_j = a_i a_j \mbox{ for }i<j \rangle.
%\end{eqnarray*}
In accordance with \cite{MR2561084}, let
\begin{eqnarray*}
\pi \colon \W(A) \to \mathrm{K}_n
\end{eqnarray*}
denote the canonical evaluation epimorphism.

\begin{defi}[\cite{MR2821178}]
Let $w\in\W(A)$. A subword of $w$ of the form $a_i u a_i$,
 where $a_i\in A$ and $u\in \W(A)$, is \emph{special} if $u$ contains both some $a_j,j>i,$ and
some $a_k,k<i$.
\end{defi}
\begin{oss}\label{O: no 1-special}
Notice that a subword $a_i u a_i$ cannot be special if $i=1$ or $i=n$.
\end{oss}

Let us recall the following fact.
\begin{teo}[\cite{MR2561084}]\label{T: Kudryavtseva}
Let $w \in \W(A)$.
The set $\pi^{-1}\pi(w)$ contains a unique element whose only subwords of the form $a_i u a_i$ are special.
\end{teo}

\begin{oss}\label{O: simplifying sequences}
\mbox{}
\begin{itemize}
\item[$(i)$] The unique element described in Theorem \ref{T: Kudryavtseva} contains at most one occurrence of $a_1$ and at most one occurrence of $a_n$.
\item[$(ii)$] In order to prove Theorem \ref{T: Kudryavtseva}, the authors of \cite{MR2561084} define a binary relation $\to$ in
$\W(A)$ as follows: $w\to v$ if and only if either
\begin{enumerate}
	\item[($\stackrel{1}{\rightarrow}$)] $w=w_1 a_i a_i w_2$, $v=w_1 a_i w_2$, or
	\item[($\stackrel{2}{\rightarrow}$)] $w=w_1 a_i u a_i w_2$, $v=w_1 a_i u w_2$ and $u\in \langle a_{i+1}, \dots, a_n \rangle $, or
    \item[($\stackrel{3}{\rightarrow}$)] $w=w_1 a_i u a_i w_2$, $v=w_1  u a_i w_2$ and $u\in \langle a_1, \dots, a_{i-1}\rangle$.
\end{enumerate}

It is possible to iterate such simplifications and write (finite) sequences
\begin{eqnarray}\label{simplifying sequence}
w\to v_1 \to v_2 \dots \to v_k,
\end{eqnarray}
so that each word contains exactly one letter less than the previous and they all belong to the same fiber %%.
 with respect to $\pi$.
Moreover, each $v_i$ is a quasi-subword of $w$ and of $v_j$, for all $j<i$.

A sequence like \eqref{simplifying sequence} is called \emph{simplifying sequence} if $v_k$ is not simplifiable any further, and each $v_i \to v_{i+1}$ is called a \emph{simplifying step}. It should be stressed that a simplifying step of type $1$ can be seen as both of type $2$ and of type $3$.

\item[$(iii)$] According to Theorem \ref{T: Kudryavtseva}, every simplifying sequence for $w$ ends on the same word.
\end{itemize}
\end{oss}

We will refer to any word, whose only subwords of the form $a_i u a_i$ are special, as a \emph{canonical word}, so that the above theorem claims existence of a unique canonical word $v$ in each $\pi^{-1}\pi(w), w \in \W(A)$.
Thus, the assignment $w \mapsto v$ is a well defined map $\Can: \W(A) \to \W(A)$ associating with each word its unique \emph{canonical form}. Notice that $w$ is canonical if and only if $w = \Can w$.

\begin{oss}\label{O: Can(a_hu)=a_hCan(u)}
\mbox{}
\begin{enumerate}
	\item[$(i)$] If $w$ is canonical then all of its subwords are canonical.
	\item[$(ii)$] A word $u\in \langle a_i, h \leq i\leq k \rangle$ is canonical if and only if, for every $j<h$ or
$j>k$, the word $a_j u$ is canonical. Moreover, in such cases,
\[
\Can(a_ju)=a_j \Can u.
\]
\end{enumerate}
\end{oss}

\begin{oss}\label{O: forma canonica}
	More consequences of Theorem \ref{T: Kudryavtseva} are that
\begin{itemize}
\item[$(i)$] the canonical form $\Can w$ is the word of minimal length in $\pi^{-1}\pi(w)$;

\item[$(ii)$] due to Remark \ref{O: simplifying sequences}$(iii)$, $\Can w$ is the last element in every simplifying sequence starting from any of the words in  $\pi^{-1}\pi(w)$;

\item[$(iii)$] $\Can w$ is a quasi-subword of all the words in any simplifying
		 sequence beginning from $w$;
	
\item[$(iv)$] if $w$ contains any given letter, so does $\Can w$.
\end{itemize}
\end{oss}

\begin{lem}\label{L: Can(uv)=Can(Can(u)Can(v))}
For every $u,v\in \W(A)$,
\begin{eqnarray*}
\Can(uv)& = & \Can\left((\Can u) v \right)\\
        & = & \Can\left(u \Can v\right)\\
        & = & \Can\left(\Can u \Can v\right)\\
\end{eqnarray*}
\end{lem}
\begin{proof}
$\Can$ is constant on fibres of $\pi$.
\end{proof}

\begin{defi}%(\cite{MR2561084})
If $I\subseteq A$, then
\[
\Can_I w=\Can \de_{A\setminus I}w,
\]
where $w\in \W(A)$. In particular $\Can_A w = \Can w$.
\end{defi}

Before proceeding further, we need to make an important observation. Say we have a sequence of steps leading from a word $y$ to a word $x$ and that each step removes a single letter. The same procedure can be applied to every subword of $y$, and again at each step (at most) a single letter is removed, eventually yielding a subword of $x$. Every subword of $x$ is obtained in this way from some (possibly non-unique) subword of $y$. Notice that in the cases we will deal with, some of the steps could be simplifications of type $\stackrel{1}{\rightarrow}$ for which there is an ambiguity on which of the two identical letters is to be removed.

We will say that a subword of $x$ of the form $a_i u a_i$ {\em originates} from a subword $w$ of $y$ if $w$ yields $a_i u a_i$ under the sequence of simplifications and $w$ is of the form $a_i v a_i$.
Let us clarify things with an example. In the following sequence of steps, we have highlighted the letter to remove at each step:
$$bdbc\,\widehat{d}\,abcdc \to b\,\widehat{d}\,bcabcdc \to \widehat{bb}\,cabcdc \to bcabcdc.$$
Notice that in the last step, there is an ambiguity on which letter is being removed.
Then the subwords $bdbcdab, dbcdab, bcdab$ of $bdbcdabcdc$ all yield $bcab$; however, $bcab$ originates only from $bdbcdab$ and $bcdab$.

\vspace{.2cm}
Henceforth, we will shorten the notation and denote by $\de_i, \T_i, \dots$ the maps $\de_{a_i}, \T_{a_i}, \dots$.

\begin{lem}\label{L: solo sottoparole di tipo aba}
Assume $y=\Can y$ and let $x=\de_{\{1,\dots, k-1\}}y$, where $1\leq k \leq n$.
Then any subword $a_i u a_i\qsubw x$ is either special or satisfies $u\in \langle a_i, \dots, a_n\rangle$.
\end{lem}
\begin{proof}
Let $a_i u a_i$ be a subword of $x$. Then $i\geq k$, as $x = \de_{\{1,\dots, k-1\}} y$ contains no $a_i, i < k$.

The above subword $a_i u a_i$ originates from a subword $a_i v a_i$ of $y$. If $a_i u a_i$ is not special, then either $u \in \langle a_j, j \leq i\rangle$, or $u \in \langle a_j, j \geq i\rangle$. However, the former case does not occur, otherwise $a_i v a_i$ would fail to be special, as $x$ is obtained from $y$ by only removing letters $a_j, 1 \leq j < k$, and $k \leq i$. This is a contradiction, as $y$ is canonical.
\end{proof}

\begin{lem}\label{L: sempl-solo-a-dx}
Assume $y=\Can y$ and let $x=\de_{\{1,\dots, k-1\}}y$, where $1\leq k \leq n$.
Then any simplifying sequence
\[
x=\de_{\{1,\dots, k-1\}}y \rightarrow \dots \rightarrow\Can x=\Can_{\{k,\dots, n\}}y
\]
is such that all simplifying steps are of type $\stackrel{2}{\rightarrow}$ (possibly of type $\stackrel{1}{\rightarrow}$).
\end{lem}
\begin{proof}
By Lemma \ref{L: solo sottoparole di tipo aba}, any subword $a_i u a_i\qsubw x$ is either special or satisfies $u\in \langle a_j, i \leq j \leq n\rangle$, hence the first simplifying step is of type $\stackrel{2}{\rightarrow}$ (and possibly of type $\stackrel{1}{\rightarrow}$).

We want to show that, starting from $x=\de_{\{1,\dots, k-1\}}y$ and applying simplifications of type $\stackrel{2}{\rightarrow}$ or of type $\stackrel{1}{\rightarrow}$, one can never obtain a word admitting a subword $a_i u a_i, u \neq \star,$ on which it is possible to apply a step of type $\stackrel{3}{\rightarrow}$ (which is not of type $\stackrel{1}{\rightarrow}$). %Notice that in such a case, $u\in \langle a_j, k \leq j < i\rangle$, so that $i >k$.

Assume therefore by contradiction that $a_i u a_i$, where $\star \neq u\in \langle a_j, k \leq j < i\rangle$, occurs as a subword after having performed some simplifying steps of type $\stackrel{2}{\rightarrow}$ on $x$. The subword $a_i u a_i$ originates from a subword $a_i v a_i$ of $x$, which is necessarily special, since $v$ cannot lie in $\langle a_j, i \leq j \leq n\rangle$, as some $a_j, j<i,$ occurs in $u \neq \star$, and $u$ is obtained from $v$ after removing some letters. This shows that letters $a_j, j>i,$ must occur in $v$, whereas they do not occur in $u$.

However, a simplifying step of type $\stackrel{2}{\rightarrow}$ that removes a letter from between the two occurrences of $a_i$ either occurs completely between them, or begins on the left of both. In the former case, it does not change the speciality of the subword, whereas in the latter case a simplification of type $\stackrel{2}{\rightarrow}$ can only remove a letter $a_j, j <i$, due to the presence of the left $a_i$; we thus obtain a contradiction.
\end{proof}
\begin{cor}\label{startsbyu}
If $u a_1 v$ is canonical, then $\Can(uv)$ admits $u$ as a prefix.
\end{cor}
\begin{proof}
By Lemma \ref{L: sempl-solo-a-dx}, all simplifying sequences from $uv = \de_1 (u a_1 v)$ to $\Can(uv)$ only contain steps of type $\stackrel{2}{\rightarrow}$ and no such simplifying step does alter $u$.
\end{proof}
\begin{prop}\label{P: u a_1 Can[k,n]v}
Assume that $u a_j v$ is canonical, and that $u \in \langle a_i, k \leq i \leq n\rangle$, where $k > j$. Then $u a_j \Can_{\{k, \dots, n\}} v$ is also canonical.

In particular, if $u a_1 v$ is canonical, then $u a_1 \Can_{\{k,\dots, n\}}v$ is also canonical.
\end{prop}
\begin{proof}
We prove the latter claim, as the proof of the former more general statement is completely analogous.

We know that $v$ is canonical. Then Lemma \ref{L: sempl-solo-a-dx} shows that $u a_1 \de_{\{1,\dots, k-1\}}v$ can be simplified into $u a_1 \Can_{\{k,\dots, n\}}v$ by only using steps of type $\stackrel{2}{\rightarrow}$ on the right of $a_1$.

If $u a_1 \Can_{\{k,\dots, n\}} v$ is not canonical, then we may find a subword $a_i x a_i$ that is not special. As both $u$ and $\Can_{\{k,\dots, n\}} v$ are canonical, then $a_1$ must occur in $x$. Say that $a_i x a_i$ originates from the subword $a_i y a_i$ of $u a_1 v$. No simplification of type $\stackrel{2}{\rightarrow}$, when performed on the right of $a_1$, can change the set of letters that appear between the two $a_i$. This yields a contradiction, as $a_i y a_i$ is special, whereas $a_i x a_i$ is not.
\end{proof}

\begin{cor}\label{febbre}
For every choice of $u, u', v, v' \in \langle a_3, a_4, \dots, a_n\rangle$,
\begin{itemize}
\item[$(i)$]
if $u a_1 v a_2 v'$ is canonical, then $u a_1 \Can(v v')$ is canonical;
\item[$(ii)$]
if $u a_2 v a_1 v'$ is canonical, then $u a_2 \Can(v v')$ is canonical;
\item[$(iii)$]
if $u a_2 u' a_1 v a_2 v'$ is canonical, then both $u a_2 u' a_1 \Can(v v')$ and $u a_2 \Can(u' v v')$ are canonical.
\end{itemize}
\end{cor}
\begin{proof}
$(i)$ follows directly from Proposition \ref{P: u a_1 Can[k,n]v}, whereas $(iii)$ follows by applying Proposition \ref{P: u a_1 Can[k,n]v} and then $(ii)$.

However, $(ii)$ is equivalent to $(i)$, as both $u a_1 v a_2 v'$ and $u a_2 v a_1 v'$ contain single occurrences of $a_1$ and $a_2$, and every simplifying sequence for the former can be turned into a simplifying sequence for the latter by switching $a_1$ with $a_2$.
\end{proof}

%%%%%%%%%%%%%%%%
%%%%%%%%%%%%%%%%
%%%%%%%%%%%%%%%%
\begin{comment}
\begin{lem}\label{L: Can_{[i,n]}(T_i(w))=a_i Can_{[i+1,n]}(T_i(w))}
Let $w\in \W(A)$, $i \in [1, n]$. Then
\begin{itemize}
\item[$(i)$] there is at most one occurrence of $a_i$ in $\Can_{[i,n]}w$;
\item[$(ii)$] if such an occurrence exists, then
\[
\Can_{[i,n]}\T_i w =a_i \Can_{[i+1,n]}\T_i w .
\]
\end{itemize}
\end{lem}
\begin{proof}
Without loss of generality, we may assume that $i=1$.
\begin{itemize}
\item[$(i)$]
%%By Remark \ref{O: no 1-special}, $\Can w$ has no special subwords of the form $a_1 u a_1$, hence it contains at most one occurrence of $a_1$.
Recall Remark \ref{O: simplifying sequences}$(i)$.
\item[$(ii)$]
As $\Can\T_1 w $ contains exactly one occurrence of $a_1$, this must be its head as, by Lemma \ref{L: sempl-solo-a-dx}, every simplifying step removing an $a_i$ must necessarily be of type $\stackrel{2}{\rightarrow}$ or $\stackrel{1}{\rightarrow}$.
Recalling Remark \ref{O: Can(a_hu)=a_hCan(u)}$(ii)$ and the definition of $\Can_{I}$,
\begin{eqnarray*}
\Can\T_1 w   & = & \Can (a_1 \de_1\T_1 w )\\
                        & = & a_1 \Can (\de_1\T_1 w )\\
                        & = & a_1\Can_{[2,n]}\T_1 w.
\end{eqnarray*}
\end{itemize}
\end{proof}
\end{comment}
%%%%%%%%%%%%%%%%
%%%%%%%%%%%%%%%%
%%%%%%%%%%%%%%%%

\begin{lem}\label{L: can-trunc}
Let $w\in \W(A)$, $1 \leq i \leq n$. Then
\[
\T_i\Can_{\{i,\dots, n\}}w=\Can_{\{i,\dots, n\}}\T_i w .
\]
\end{lem}
\begin{proof}
Once again, we may assume without loss of generality that $i = 1$.
If there are no occurrences of $a_1$ in $w$, then both sides equal the empty word and we are done.

Otherwise, using Remark \ref{O: w=w'T_i(w)}$(ii)$, write $w = u \T_1 w$ and $\Can \T_1 w = a_1 v$. We are asked to show that
\[
\T_1 \Can (u \T_1 w) = \Can (\T_1 w),
\]
which is equivalent, using Lemma \ref{L: Can(uv)=Can(Can(u)Can(v))}, to
\[
\T_1 \Can (u a_1 v) = a_1 v.
\]
Notice that $v$ is canonical and $u \in \langle a_2, \dots, a_n\rangle$. The simplifying steps that can occur on a word of type $u a_1 v$ may only affect $u$: indeed, $v$ is canonical, there is an only occurrence of $a_1$ in $u a_1 v$, and the only special words that begin in $u$ and end in $v$ contain an occurrence of $a_1$, thus leading to a $\stackrel{3}{\rightarrow}$.

An easy induction now shows that $\Can(u a_1 v) = u' a_1 v$, where $u'$ is a quasi-subword of $u$, hence $\T_1 \Can(u a_1 v) = a_1 v$.
\end{proof}
%%%%%%%%%%%%%%%%
%%%%%%%%%%%%%%%%
%%%%%%%%%%%%%%%%
\begin{comment}
\begin{lem}\label{L: Can(uv) preceq Can(ua_1 v)}
For every $u,v\in \langle a_2, \dots, a_n \rangle $,
\[
\Can(uv)\qsubw \Can(ua_1 v)
\]
\end{lem}
\begin{proof}
We may assume using Lemma \ref{L: Can(uv)=Can(Can(u)Can(v))} that $v$ is canonical.

Argue as in the previous proof, to show that $\Can(u a_1 v) = u' a_1 v$, where $u'$ is obtained by a sequence of simplifying steps that either take place on the left of $a_1$, or are of type $\stackrel{3}{\rightarrow}$. Every such step can be also performed starting with the word $uv$, as the presence or absence of $a_1$ has no influence on steps of type $\stackrel{3}{\rightarrow}$ that begin on $u$ and end on $v$. Thus, $u'v$ can be obtained from $uv$ by a simplifying sequence, hence $\Can(uv) = \Can(u'v)$.
Now,
\[
\Can(uv) = \Can(u'v) \qsubw u'v \qsubw u' a_1 v = \Can(u a_1 v),
\]
and we are done.
\end{proof}
An immediate consequence of the last lemma is that if $w$ contains only one occurrence of $a_1$, then $\Can_{[2,n]}w\qsubw \Can w$. However, this fact holds without the single occurrence assumption.
\end{comment}
%%%%%%%%%%%%%%%%
%%%%%%%%%%%%%%%%
%%%%%%%%%%%%%%%%

\begin{lem}\label{L: Can_I(w)<Can(w)}\label{L: Can_I(Can)=Can_I}
Let $w\in \W(A)$, $1 \leq k \leq n$. Then
\begin{itemize}
\item[$(i)$]
$\Can_{\{k,\dots, n\}}w = \Can_{\{k,\dots, n\}} \Can w$;
\item[$(ii)$]
$\Can_{\{k,\dots, n\}}w$ is a quasi-subword of $\Can w$.
\end{itemize}
%As a consequence,  $\Can_{[k,n]}w$ is a quasi-subword of $\Can w$ for every $k\leq n$.
\end{lem}
\begin{proof}
Let $k>1$, the case $k = 1$ being trivial. It is easy to check that $\Can \circ \de_{\{1,\dots, k-1\}}$ is constant on fibres of $\pi$ --- it is invariant under all simplifying steps --- which takes care of $(i)$. We may therefore assume in $(ii)$ that $w$ is a canonical word; however $\Can_{\{k,\dots, n\}}w$ is obviously a quasi-subword of $w$.
%%%%%%%%%%%%%%%%
%%%%%%%%%%%%%%%%
%%%%%%%%%%%%%%%%
\begin{comment}

	If $w\in \langle a_2,\dots, a_n \rangle $, then $w=\de_1 w$ and there is nothing to prove.

Otherwise, using Remark \ref{O: w=w'T_i(w)}$(ii)$, write $w=w'\T_1 w $. According to Lemma \ref{L: Can(uv)=Can(Can(u)Can(v))} and Lemma
\ref{L: Can_{[i,n]}(T_i(w))=a_i Can_{[i+1,n]}(T_i(w))}$(ii)$,
\begin{eqnarray*}\label{w}
\Can w& = & \Can\left(w'\,  \T_1 w\right)\\
        & = & \Can\left(w'\, \Can\T_1 w \right)\\
        & = & \Can\left(w'\, a_1 \Can_{[2,n]}\T_1 w \right).
\end{eqnarray*}

On the other hand, by Remark \ref{O: w=w'T_i(w)}$(ii)$ and Lemma \ref{L: Can(uv)=Can(Can(u)Can(v))} (and recalling that here $\de_1 w'=w'$),
\begin{eqnarray*}\label{de_1(w)}
	\Can_{[2,n]}w & = & \Can \de_1(w'\T_1 w)\\
								& = & \Can\left(w'\, \de_1\T_1 w \right)\\
                & = & \Can\left(w'\, \Can_{[2,n]}\T_1 w \right).
\end{eqnarray*}
We can now apply Lemma \ref{L: Can(uv) preceq Can(ua_1 v)}.
The latter statement is taken care of by an easy induction.
\end{comment}
%%%%%%%%%%%%%%%%
%%%%%%%%%%%%%%%%
%%%%%%%%%%%%%%%%
\end{proof}

\section{The join operation}\label{join}
Given an update system over the complete oriented graph $\Gamma_n$, Proposition \ref{P: I, II, III} proves that its dynamics monoid is an epimorphic image of Kiselman's semigroup $\mathrm{K}_n$.
In next section, we will exhibit an update system $\S_n^\star$ over $\Gamma_n$ whose dynamics monoid is isomorphic to $\mathrm{K}_n$; from a dynamical point of view, $\S_n^\star$ serves as a {\em universal update system}.

We introduce the following operation in order to construct, later, a family of update functions.
\begin{defi}
Take $u,v\in \W(A)$. The \emph{join} of $u$ and $v$, denoted by $[u,v]$, is the shortest word admitting $u$
as quasi-subword and $v$ as suffix. Namely,
\[
[u,v]=u^+v
\]
where $u=u^+u^-$, so that $u^-$ is the longest suffix of $u$ which is a quasi-subword of $v$. \end{defi}
Notice that the decomposition $u = u^+ u^-$ strictly depends on the choice of $v$.
\begin{ex}
For instance, consider $u=cbadc$ and $v=abdc$. Then,
\[
[u,v]= cbabdc.
\]
\end{ex}
\begin{oss}\label{O: propr di [u,v]}
\mbox{}
\begin{itemize}
	\item[$(i)$]  The empty word $\star$ is a subword of every $w \in \W(A)$, so that
\[
[\star, u]=[u,\star]=u.
\]
\item[$(ii)$] If $u$ is a quasi-subword of $v$, then $[u,v]=v$.
\item[$(iii)$] If $u$ is not a quasi-subword of $v$, then
\[
[wu,v]=w u^+v=w[u,v],
\]
for every $w\in \W(A)$.
\item[$(iv)$] If $[u,v]=u^+v$, then
\[
[u,wv]=[u^+,w]v
\]
Indeed, write $u=u^+u^-$: $u^-$ is the longest suffix of $u$ which is a quasi-subword of $v$, but there could be a suffix
of $u^+$ which is a quasi-subword of $w$.
\item[$(v)$] If $u, v\in \W(A)$ are canonical, $[u,v]$ may fail to be so. For instance, $[a_1 a_2, a_2 a_1]=a_1 a_2 a_1$, which is not canonical.
\end{itemize}
\end{oss}
The following lemma will be used later in the proof of Proposition \ref{P: main-theorem_for Gamma_n}.
\begin{lem}\label{L: ux qsottop uy => x qsottop y}
Let $u,x,y\in \W(A)$, if $ux \qsubw uy$, then $x \qsubw y$.
\end{lem}
\begin{proof}
As $ux$ is a quasi-subword of $uy$, by Remark \ref{O: propr di [u,v]}$(ii)$,
\[
[ux,uy]=uy.
\]
Assume that $x$ is not a quasi-subword of $y$ and write it as $x=x^+x^-$, where $x^-$ is its longest suffix which is a quasi-subword of $y$. Using Remark \ref{O: propr di [u,v]}$(iii)$,
\[
[ux,y]=ux^+y.
\]
Finally, using Remark \ref{O: propr di [u,v]}$(iv)$,
\[
[ux,uy] = [ux^+, u]y
\]
However, $[ux^+, u]=u$ can hold only if $u x^+$ is not longer than $u$, %% Ho l'impressione che "no longer than" abbia solo significato temporale, ma non so come cambiarla
i.e., only if $x^+ = \star$, %and
hence $x$ is a quasi-subword of $y$.
\end{proof}

\section{An update system with universal dynamics}\label{sstella}

If $U, V \subset \W(A)$, denote by $[U, V]\subset \W(A)$ the subset of all elements $[u, v], u\in U, v \in V$.
\begin{defi}
The update system $\S_n^\star$ is the triple $(\Gamma_n,S_i,f_i)$, where
\begin{enumerate}
\item $\Gamma_n$ is, as before, the complete oriented graph on $n$ vertices, where $i\to j$ if and only if $i < j$.
\item  On vertex $i$, the state set $S_i \subset \W(A)$ is inductively defined as
\[
S_i=
\begin{cases}
\{\star, a_n\} & \mbox{ if } i = n,\\
\{\star, a_{n-1}, a_{n-1}a_n\} & \mbox{ if } i = n-1,\\
\{\star\} \cup a_i [S_n, [\dots, [S_{i+2}, S_{i+1}] \dots]] & \mbox{ if } 1 \leq i \leq n-2.
\end{cases}
%\{\star\}\cup a_i \langle a_{i+1}, \dots, a_n \rangle \subset \W(A).
\]
\item  On vertex $i$, the vertex function is
\begin{eqnarray*}
f_i\colon S[i]%\prod_{j=i+1}^n S_j
& \to & S_i\\
(s_{i+1},\dots, s_n)& \mapsto & a_i [s_n, [\dots, [s_{i+2},s_{i+1}]\dots]],
\end{eqnarray*}
if $i \leq n-2$, whereas $f_{n-1}(s_n) = a_{n-1} s_n$ and $f_n \equiv a_n$ is constant.

\end{enumerate}
\end{defi}
We will abuse the notation and denote by $\star$ also the system state $(\star, \dots, \star) \in S$.
The rest of the paper will be devoted to the proof of the following
\begin{prop}\label{P: main-theorem_for Gamma_n}
Consider the evaluation morphism $\FF\colon \W(A)\to D(\S_n^{\star})$, mapping each word $w\in \W(A)$ to the
corresponding evolution $\FF_w\in D(\S_n^{\star})$. If $\mathbf{p} = (p_1, \dots, p_n) = \FF_w \star$, then:
\begin{itemize}
\item[$(i)$]
One has $p_i=(\FF_w \star)_i =\Can_{\{i,\dots, n\}} \T_i w.$
\item[$(ii)$] For every choice of $k, 1\leq k \leq n$, one may find $j, 1 \leq j \leq k,$ so that
\begin{eqnarray*}
\left[p_k,\left[\dots,\left[p_2,p_1\right]\dots\right]\right]=\T_{\{1,\dots, k\}}\Can w.
\end{eqnarray*}
%and $\T_i\Can w$ is a suffix of $\T_j\Can w$ for all $i\leq k$.
\end{itemize}
\end{prop}
Our central result then follows immediately.
\begin{teo}\label{T: Can(w)=[p_n, dots [p_2, p_1]]}\label{main}
With the same hypotheses and notation as in Proposition \ref{P: main-theorem_for Gamma_n},
\begin{itemize}
\item[$(i)$] $\Can w=\left[p_n,[\dots,\left[p_2,p_1\right]\dots]\right]$;
\item[$(ii)$] if $u,v\in \W(A)$, then $\FF_u = \FF_v$ if and only if $\Can u=\Can v$;
\item[$(iii)$] $\mathrm{K}_n$ is isomorphic to $D(\S_n^{\star})$.
\end{itemize}
\end{teo}
\begin{proof}
\mbox{}
\begin{itemize}
	\item[$(i)$]
Use Proposition \ref{P: main-theorem_for Gamma_n}$(ii)$ when $k=n$. Then
$$\left[p_n,\left[\dots,\left[p_2,p_1\right]\dots\right]\right] = \T_{\{1,\dots, n\}} \Can w = \Can w.$$
% for a suitably chosen $j$. However, $\T_i \Can w$ is a suffix of $\T_j \Can w$ for all $i$, hence, by Remark \ref{O: T_a < T_b o vice versa}$(ii)$, $\T_j \Can w = \Can w$.

\item[$(ii)$] If $\FF_u = \FF_v$, then they certainly compute the same state on $\star$. However, by $(i)$, one may recover both $\Can u$ and $\Can v$ from this state, hence $\Can u = \Can v$. The other implication follows trivially, as $\Can u = \Can v$ forces $u$ and $v$ to induce the same element in $\mathrm{K}_n$, hence the same dynamics on $\S_n^*$.
\item[$(iii)$] This is just a restatement of $(ii)$: distinct elements in $\mathrm{K}_n$ have distinct $\S_n^*$-actions, since they act in a different way on the system state $\star$.
\end{itemize}
\end{proof}

\begin{oss}
As an immediate consequence of Proposition \ref{P: main-theorem_for Gamma_n}, if $j = \h(\Can w)$, then $\T_j \Can w = \Can w$, hence 
%
%n order to prove Theorem \ref{T: Can(w)=[p_n, dots [p_2, p_1]]}, we have located $j\in [1,n]$ such that
%$$\T_j \Can w = \Can w,$$
%whence $\h(\Can w) = a_j$. Now, there is no occurrence of $a_j$ in $p_{j+1}, \dots, p_n$, so that one necessarily have
\[
\left[p_j,\left[\dots,\left[p_2,p_1\right]\dots\right]\right] = \left[p_n,\left[\dots,\left[p_2,p_1\right]\dots\right]\right] =\Can w.
\]
\end{oss}\

We will prove Proposition \ref{P: main-theorem_for Gamma_n} by induction on the number $n$ of vertices. The following technical fact is needed in the proof of the inductive step, and we assume in its proof that Proposition \ref{P: main-theorem_for Gamma_n} and Theorem \ref{T: Can(w)=[p_n, dots [p_2, p_1]]} hold on $\S_k^{\star}$, for $k<n$.

\begin{prop}\label{P: technical}
Let $u, v \in \langle a_{j+1}, a_{j+2}, \dots, a_n\rangle$ be chosen so that $u a_j \Can v$ is a canonical word. Then $[\Can(uv), v] = u v.$ In particular, $[\Can(uv), a_j v] = u a_j v$.
\end{prop}
\begin{proof}
We may assume, without loss of generality, that $j = 1$.

The statement is easily checked case by case when $n=1$ or $2$.
Notice that $u$ can have at most one occurrence of $a_2$, whereas $v$ may have many. Let us therefore distinguish four cases:
\begin{enumerate}
\item There is no occurrence of $a_2$ in either $u$ or $v$.

In this case, we can use inductive assumption, after removing the vertex indexed by $2$.

\item There is a single occurrence of $a_2$ in $u$.

Write $u = u' a_2 u''$. As $u' a_2 u'' a_1 \Can v$ is canonical, then, using Corollary \ref{febbre}$(ii)$ and Lemma \ref{L: Can(uv)=Can(Can(u)Can(v))},
\begin{eqnarray*}
\Can (u' a_2 u'' v) & = & \Can(u' a_2 u'' \Can v)\\
                    & = & u' a_2 \Can(u'' \Can v) \\
                    & = & u' a_2 \Can(u'' v).
\end{eqnarray*}

Thus, we need to compute $[u' a_2 \Can(u'' v), v]$. However, applying Case 1 gives
\[
[\Can(u'' v), v] = u'' v,
\]
hence $[u' a_2 \Can(u'' v), v] = u' a_2 u'' v$ follows from Remark \ref{O: propr di [u,v]}%%$(iv)$.
$(iii)$.
\item $a_2$ occurs in $v$, but not in $u$.

Write $\Can v = v' a_2 v''$. As $ua_1 \Can v = u a_1 v' a_2 v''$ is canonical, then, using Corollary \ref{febbre}$(i)$, also $u a_1 \Can(v' v'')$ is canonical. However Lemma \ref{L: Can_I(Can)=Can_I}
informs us that
\begin{eqnarray*}
\Can(v'v'') & = & \Can\de_2 \Can v \\
            & = & \Can_{\{3,\dots, n\}} \Can v\\
            & = & \Can_{\{3,\dots, n\}} v \\
            & = & \Can \de_2 v,
\end{eqnarray*}
so that $u a_1 \Can \de_2 v$ is canonical. We may then use Case 1 to argue that
\begin{equation}\label{ustafuori}
[\Can(u \de_2 v), \de_2 v] = u \de_2 v.
\end{equation}

Lemma \ref{L: Can_I(w)<Can(w)} implies that $\Can(u \de_2 v) = \Can \de_2 (uv)$ is a quasi-subword of $\Can(uv)$. By Corollary \ref{startsbyu}, both $\Can(uv)$ and $\Can (u \de_2 v)$ admit $u$ as a prefix, so, using Lemma \ref{L: ux qsottop uy => x qsottop y}, we can write
\[
\Can(uv) = uy, \qquad \Can(u\de_2 v) = uz, \qquad z \qsubw y.
\]

We need to show that $[\Can(uv), v] = u v$. Now, applying Corollary \ref{febbre} $(ii)$,
\[
\Can(uv) = \Can(uv') a_2 v'' = ux a_2 v'',
\]
where $x$ is some quasi-subword of $v'$. As $$y = x a_2 v'' \qsubw v' a_2 v'' = \Can v \qsubw v,$$
then $[\Can(uv), v] = u^+ v$
where $u=u^+u^-$ and $u^- y \qsubw v$.
However, this would force
\begin{eqnarray*}
u^- z\qsubw u^- y \qsubw v
\end{eqnarray*}
hence also $u^- z \qsubw  \de_2 v$; and this is only possible if $u^- = \star$, as Equation \eqref{ustafuori} shows.

\item $a_2$ occurs both in $u$ and in $v$.

This is similar to the previous case.
Write $u = u' a_2 u''$, $\Can v = v' a_2 v''$. Applying Corollary \ref{febbre}$(iii)$ to the canonical word $u a_1 \Can v = u' a_2 u'' a_1 v' a_2 v''$, we obtain that both
\[
u' a_2 u'' a_1 \Can \de_2 v = u' a_2 u'' a_1 \Can(v'v'')
\]
and
\[
u' a_2 \Can(u'' \de_2 v) = u' a_2 \Can(u'' \Can \de_2 v) = u' a_2 \Can(u''v'v'')
\]
are canonical. Similarly,
\begin{eqnarray*}
\Can(uv)            & = & \Can(u' a_2 u'' v' a_2 v'') \\
                    & = & \Can(u' a_2 u'' v' v'')\\
                    & = & \Can(u' a_2 \Can(u'' v' v''))\\
                    & = & u' a_2 \Can(u'' \de_2 v).
\end{eqnarray*}
Now, as $u'' a_1 \Can \de_2 v$, being a subword of $u' a_2 u'' a_1 \Can \de_2 v$, is canonical, we have %use Case 1 and obtain
\begin{equation}\label{u''stafuori}
[\Can(u'' \de_2 v), \de_2 v] = u'' \de_2 v.
\end{equation}

As a consequence,
\begin{eqnarray*}
[\Can(u' a_2 u'' v' a_2 v''), \de_2 v]  & = & [u' a_2 \Can(u'' \de_2 v), \de_2 v]\\
                                        & = & [u' a_2, u'']\de_2 v \\
                                        & = & u' a_2 u'' \de_2 v,
\end{eqnarray*}
and one may complete the proof as in Case 3.
\end{enumerate}
\end{proof}

\section{Proof of Proposition \ref{P: main-theorem_for Gamma_n}}
The basis of induction $n=1$ being trivial, we assume that Proposition \ref{P: main-theorem_for Gamma_n} and Theorem \ref{T: Can(w)=[p_n, dots [p_2, p_1]]} hold for a complete graph on
less than $n$ vertices,

\begin{proof}[Proof of Proposition \ref{P: main-theorem_for Gamma_n}]
Let us start by proving Part $(i)$.

Let $w\in \W(A)$.
By inductive hypothesis, for a complete graph on $n-1$ vertices $2, \dots, n$, the SDS map $\FF_{\de_1 w }$ constructs on the $i$th vertex the state
\[
p_i=\Can_{\{i,\dots, n\}} \T_i\de_1 w = \Can_{\{i, \dots, n\}} \T_i w,
\]
%When $i>1$,
%\[
%\de_1 w \sim_i w
%\]
%and we may use induction to argue that $p_i = \Can_{[i, n]} \T_i \de_1 w$. %%qui mancava lo slash davanti a  \T_i E IN PIU' STAVAMO RIPETENDO DUE VOLTE LA STESSA AFFERMAZIONE!!!
%Now, Lemma \ref{L: de-trunc} allows us to exchange $\de_1$ and $\T_i$, so that
%\begin{eqnarray*}
%p_i & = & \Can_{\{i,\dots, n\}} \T_i\de_1 w\\
%    & = & \Can_{\{i,\dots, n\}} \de_1\T_i w \\
%    & = & \Can \de_{\{1,\dots, i-1\}} \de_1\T_i w \\
%    & = & \Can_{\{i,\dots, n\}} \T_i w .
%\end{eqnarray*}
hence, the statement holds for all vertices $i>1$.

As far as vertex $1$ is concerned, we need to show that $p_1 = \Can \T_1 w$.
If $w$ does not contain
the letter $a_1$, then $p_1=\star = \T_1 w$ and there is nothing to prove; otherwise, $ (\FF_w)_1 = (\FF_{\T_1 w})_1$. We know that the state on vertex $1$ depends only on the
system state $(p'_2, \dots, p'_n)$, which is computed by $\de_1\T_1 w$ on the subgraph indexed by $\{2, \dots, n\}$, i.e.,
\begin{eqnarray*}
p_1  & = & a_1 \left[p'_n,[ \dots, [p'_3,p'_2]\dots]\right].
\end{eqnarray*}
However, we may apply induction hypothesis and use Theorem \ref{T: Can(w)=[p_n, dots [p_2, p_1]]}$(i)$, which yields
\begin{eqnarray*}
p_1 & = & a_1 \left[p'_n,[ \dots, [p'_3,p'_2]\dots]\right]\\%%mancava a_1 prima di \left[...
    & = & a_1 \Can \de_1\T_1 w \\
    & = & \Can \T_1 w
\end{eqnarray*}

As for Part $(ii)$ of Proposition \ref{P: main-theorem_for Gamma_n}, we proceed by induction on $k$. The basis of induction descends directly from Part $(i)$ and Lemma \ref{L: can-trunc}, as
\begin{eqnarray*}
p_1 & = & \Can\T_1 w \\
    & = & \T_1 \Can w .
\end{eqnarray*}

Assume now $k > 1$.  By inductive hypothesis, there exists $j < k$ such that
\[
\left[p_{k-1},\left[\dots,\left[p_2,p_1\right]\dots\right]\right]=\T_{\{1,\dots, k-1\}}\Can w.
\]
Recall that, by Remark \ref{O: T_a < T_b o vice versa}%$(i)$
, either $\T_k \Can w$ is a suffix of $\T_{\{1,\dots, k-1\}} \Can w$ or vice versa. In the former case, we know by Part $(i)$ and Lemma \ref{L: can-trunc}, that
\begin{eqnarray}\label{p_k-T-Can}
p_k  =  \Can_{\{k,\dots, n\}}\T_k w=\T_k\Can_{\{k,\dots, n\}}w.
\end{eqnarray}
By Lemma \ref{L: Can_I(w)<Can(w)}, $\Can_{\{k,\dots, n\}}w$ is a quasi-subword of $\Can w$, hence $p_k = \T_k\Can_{\{k,\dots, n\}}w$ %%tolto le parentesi tonde intorno a tra T_k e Can
is a quasi-subword of $\T_k\Can w$, which is a suffix of $\T_{\{1,\dots, k-1\}}\Can w$. Thus,
\begin{eqnarray*}
\left[p_k,\left[p_{k-1}\left[\dots,\left[p_2,p_1\right]\dots\right]\right]\right] & = &  [p_k,\T_{\{1,\dots, k-1\}}\Can w]\\
                                                                            & = &  \T_{\{1,\dots, k-1\}}\Can w = \T_{\{1,\dots, k\}}\Can w.
\end{eqnarray*}
\begin{comment}
If, instead, $\T_j\Can w$ is a suffix of $\T_k\Can w$, we argue as follows. Choose $u$
such that $\T_k\Can w = u\T_j\Can w.$ As $u\in \langle a_i, k\leq i \leq n\rangle$, then
$$p_k = \Can_{[k,n]} u \T_j \Can w = \Can (u \de_{[1, k-1]} \T_j w).$$

By Proposition \ref{P: u a_1 Can[k,n]v}, we have that $u a_j \Can_{[k,n]} \T_j w$ is a canonical word. Recall that
\begin{eqnarray*}
\Can \de_{[1,k-1]} \T_j \Can w & = & \Can_{[k,n]} \T_j \Can w =\\
                                & = & \Can_{[k,n]} \Can \T_j w = \\
                                & = & \Can_{[k,n]} \T_j w.
\end{eqnarray*}
Applying now Proposition \ref{P: technical} to the word $u a_j \de_{[1, k-1]} \T_j w$, one obtains
\begin{eqnarray*}
[p_k, a_j \de_{[1,k-1]} \T_j \Can w] & = & [\Can(u \de_{[1, k-1]} \T_j \Can w), a_j \de_{[1, k-1]} \T_j \Can w] =\\
                                        & = & u a_j  \de_{[1, k-1]} \T_j \Can w,
\end{eqnarray*}
hence a fortiori $[p_k, \T_j \Can w] = u \T_j \Can w$.

\end{comment}

If, instead, $\T_{\{1,\dots, k-1\}}\Can w$ is a suffix of $\T_k\Can w$, pick $j$ so that $\T_j \Can w = \T_{\{1,\dots, k-1\}}\Can w$ and argue as follows.

Choose $u$
such that $\T_k\Can w = u\T_j\Can w.$ As $u\in \langle a_i, k\leq i \leq n\rangle$, then by Part $(i)$
\[
%%p_k  =  \Can_{[k,n]} u \T_j \Can w =  \Can (u \de_{[1, k-1]} \T_j w).%%qui mancano le parentesi nel secondo termine e il \Can nel terzo
p_k  =  \Can_{\{k,\dots, n\}} (u \T_j \Can w) =  \Can (u \de_{\{1, \dots, k-1\}} \T_j \Can w).
\]
By Proposition \ref{P: u a_1 Can[k,n]v}, as $\T_k\Can w= u \T_j \Can w$ is canonical, we may argue that the word %%$u a_j \Can_{[k,n]} \T_j w$
$u a_j \Can_{\{k,\dots, n\}} \T_j \Can w$ is also canonical. 
%%Recall that
%%\begin{eqnarray*}
%%\Can \de_{[1,k-1]} \T_j \Can w & = & \Can_{[k,n]} \T_j \Can w =\\
%%                                & = & \Can_{[k,n]} \Can \T_j w = \\
%%                                & = & \Can_{[k,n]} \T_j w.
%%\end{eqnarray*}
Apply now Proposition \ref{P: technical} to  %%$u a_j \de_{[1, k-1]} \T_j w$, one obtains
$u a_j \de_{\{1, \dots, k-1\}} \T_j \Can w$, to obtain
\begin{eqnarray*}
[p_k, a_j \de_{\{1, \dots k-1\}} \T_j \Can w] & = & [\Can(u \de_{\{1, \dots, k-1\}} \T_j \Can w), a_j \de_{\{1, \dots, k-1\}} \T_j \Can w] =\\
                                        & = & u a_j  \de_{\{1, \dots, k-1\}} \T_j \Can w,
\end{eqnarray*}
hence a fortiori $[p_k, \T_j \Can w] = u \T_j \Can w = \T_k \Can w$, which equals $\T_{\{1, \dots, k\}} \Can w$.

\end{proof}
\section{Conclusions and further developments}\label{conclude}
If $\Gamma$ is a finite directed acyclic graph, then any update system $\S$ supported on $\Gamma$ induces a dynamics monoid $D(\S)$ which naturally arises as a quotient of $\HK_\Gamma$. We refer to the smallest quotient of $\HK_\Gamma$ through which all evaluation maps $\FF: \W(V) \to D(\S)$ factor as the {\em universal dynamics monoid} $D(\Gamma)$. In this paper, we have proved that $D(\Gamma_n) \simeq \HK_{\Gamma_n} = \mathrm{K}_n$.
\begin{conj}
$D(\Gamma) \simeq \HK_\Gamma$ for every finite directed acyclic graph.
\end{conj}
This has been computationally checked for all instances with $\leq 4$ vertices and on most graphs on $5$ vertices. A conceptual approach to the problem of determining the universal dynamics of a given graph $\Gamma$ is by constructing an initial object, if there exists one, in the category of (pointed) update systems supported on $\Gamma$; here ``pointed'' means that a preferred system state has been chosen.
\begin{conj}
The pair $(\S_n^\star, \star)$ is an initial object in the category of pointed update systems supported on $\Gamma_n$.
\end{conj}
The dynamics of an initial update system is clearly universal and its elements are told apart by their action on the marked system state $\star$: this has been the philosophy underneath the proof of Theorem \ref{T: Can(w)=[p_n, dots [p_2, p_1]]}.

%\begin{comment}
\section*{Acknowledgements}
We would like to thank Francesco Vaccarino for his interest in the problem, Volodymyr Mazorchuk and Riccardo Aragona for giving a preliminary assessment of this work, Henning Mortveit for encouraging comments and Giorgio Ascoli for indicating possible applications. We are grateful to the unknown referees for their suggestions, which were decisive in order to improve and substantially simplify the exposition.

Results in this paper are part of EC's doctoral dissertation, which was partially written during a leave of absence from Banca d'Italia. This work was completed while ADA was visiting University of Pisa on a sabbatical leave; he was supported by Ateneo Fundings from ``La Sapienza'' University in Rome.
%\end{comment}

%\addcontentsline{toc}{section}{\refname}
\providecommand{\bysame}{\leavevmode\hbox to3em{\hrulefill}\thinspace}
\providecommand{\MR}{\relax\ifhmode\unskip\space\fi MR }
\providecommand{\MRhref}[2]{%
  \href{http://www.ams.org/mathscinet-getitem?mr=#1}{#2}
}
\providecommand{\href}[2]{#2}

%\bibliographystyle{amsalpha}	
%\bibliography{bibliography}	

\begin{thebibliography}{BMR01}

\bibitem[AD13]{small}
Riccardo Aragona and Alessandro D'Andrea, \emph{Hecke-Kiselman monoids of
  small cardinality}, Semigroup Forum \textbf{86} (2013), no.~1, 32--40.
  %\MR{3016259}

\bibitem[BMR00]{BMR1}
C.~L. Barrett, H.~S. Mortveit, and C.~M. Reidys, \emph{Elements of a theory of
  simulation. {II}. Sequential dynamical systems}, Appl. Math. Comput.
  \textbf{107} (2000), no.~2-3, 121--136. %\MR{1724827 (2000k:68170)}

\bibitem[BMR01]{BMR2}
\bysame, \emph{Elements of a theory of simulation. {III}. Equivalence of
  SDS}, Appl. Math. Comput. \textbf{122} (2001), no.~3, 325--340. %\MR{1842611 (2002e:68141)}

\bibitem[BR99]{BR}
C.~L. Barrett and C.~M. Reidys, \emph{Elements of a theory of computer
  simulation. {I}. Sequential CA over random graphs}, Appl. Math. Comput.
  \textbf{98} (1999), no.~2-3, 241--259. %\MR{1660115 (99h:68181)}

\bibitem[For12]{forsberg}
Love Forsberg, \emph{Effective representations of Hecke-Kiselman monoids of
  type A}, ArXiv:1205.0676.

\bibitem[GM11]{MR2821178}
Olexandr Ganyushkin and Volodymyr Mazorchuk, \emph{On Kiselman quotients of
  0-Hecke monoids}, Int. Electron. J. Algebra \textbf{10} (2011), 174--191.
  %\MR{2821178 (2012g:20121)}

\bibitem[GM13]{grensingmazorchuk}
Anna-Louise Grensing and Volodymyr Mazorchuk, \emph{Monoid algebras of
  projection functors}, Semigroup Forum (2013).

\bibitem[Gre51]{MR0042380}
J.~A. Green, \emph{On the structure of semigroups}, Ann. of Math. (2)
  \textbf{54} (1951), 163--172. %\MR{0042380 (13,100d)}

\bibitem[Gre12]{grensing}
Anna-Louise Grensing, \emph{Monoid algebras of projection functors}, J. Algebra
  \textbf{369} (2012), 16--41. %\MR{2959784}

\bibitem[Kis02]{MR1881029}
Christer~O. Kiselman, \emph{A semigroup of operators in convexity theory},
  Trans. Amer. Math. Soc. \textbf{354} (2002), no.~5, 2035--2053. %\MR{1881029  (2003a:52008)}

\bibitem[KM09]{MR2561084}
Ganna Kudryavtseva and Volodymyr Mazorchuk, \emph{On Kiselman's semigroup},
  Yokohama Math. J. \textbf{55} (2009), no.~1, 21--46. %\MR{2561084  (2010k:20097)}

\bibitem[MR08]{MR2357144}
Henning~S. Mortveit and Christian~M. Reidys, \emph{An introduction to
  sequential dynamical systems}, Universitext, Springer, New York, 2008.
  %\MR{2357144 (2009d:37002)}

\end{thebibliography}
\vfill	
%\bibitem{gm} O. Ganyushkin, V. Mazorchuk,  \textit{On Kiselman quotients of 0-Hecke monoids}, Int. Electron. J. Algebra, 10 (2011), 174-191.
%\bibitem{km} G. Kudryavtseva and V. Mazorchuk, \textit{On Kiselman's semigroup}, Yokohama Math. J., 55(1) (2009), 21-46.
%\bibitem{rs} R. Richardson, T. Springer,  \textit{The Bruhat order on symmetric varieties}. Geom. dedicata 35 (1990), 389-436.
\end{document}